\documentclass[11pt]{article}
\usepackage[a4paper,margin=30mm]{geometry}
\usepackage{amsmath,amssymb,amsthm,mathtools,bm}
\usepackage{enumitem}
\usepackage{microtype}
\usepackage{hyperref}
\hypersetup{hidelinks}
\usepackage{booktabs}
\usepackage{graphicx}

\numberwithin{equation}{section}
\newtheorem{theorem}{Theorem}[section]
\newtheorem{proposition}[theorem]{Proposition}
\newtheorem{lemma}[theorem]{Lemma}
\newtheorem{corollary}[theorem]{Corollary}
\newtheorem{remark}[theorem]{Remark}
\newtheorem{assumption}[theorem]{Assumption}
\newcommand{\E}{\mathbb E}
\newcommand{\R}{\mathbb R}

\newcommand{\norm}[1]{\lVert #1\rVert}

\title{\bfseries Third-moment cancellation and predictable representation
in finite-state BSDE approximation}
\author{Akihiro Tanaka\thanks{Sumitomo Mitsui Banking Corporation, 1-1-2 Marunouchi, Chiyoda-ku, Tokyo 100-0005, Japan. Email: \texttt{tnkaki2000@gmail.com}.}}
\date{}

\begin{document}
\maketitle

\begin{abstract}
We study finite-state approximations of Brownian-driven Markovian backward
stochastic differential equations when the same one-step increment is used
both for weak propagation and for martingale representation.  For a
centered isotropic state-independent increment with vanishing third moment,
we prove first-order convergence of the Markov value field, its spatial
gradient, and the predictable \(Z\)-field under suitable smoothness assumptions.  The
corresponding orthogonal martingale has maximal \(L^2\) size
\(O(h^{1/2})\).  We identify its quadratic representation defect and prove
a joint functional limit theorem for the state process and the normalized
orthogonal martingale, which in particular yields a mixed-normal terminal limit
with path-dependent conditional variance.  On the trace-free quadratic sector,
the representation defect has the sharp universal operator lower bound
$d/[2(d+2)]$; equality holds exactly for normalized laws that are weighted
spherical $4$-designs.
In dimension \(d\ge2\), vanishing third moment is
incompatible with exact scalar one-step predictable representation for any
finite-state centered isotropic increment.  Consequently, every such
exact-representation law has a nonzero cubic direction, and a smooth
Fourier terminal condition converts this direction into a sharp global
\(h^{1/2}\) weak error.  For the centrally symmetric regular-simplex rule we
also identify the leading quadratic GKW component, compute the exact
first-order asymptotic of the accumulated orthogonal-martingale variance,
and recover it from the general functional limit.
Thus the same third-order tensor links weak accuracy and predictable
representation, yielding an accuracy--representation tradeoff for
finite-state BSDE approximations.
\end{abstract}

\medskip
\noindent\textbf{Keywords.}
backward stochastic differential equation; random-walk approximation;
predictable representation; orthogonal martingale; weak approximation;
finite-state martingale.

\medskip
\noindent\textbf{2020 Mathematics Subject Classification.}
Primary 60H10, 60H35; Secondary 60F17, 60G42, 65C30.

\section{Introduction}

Backward stochastic differential equations can be approximated by backward
equations driven by discrete-time martingales or random walks.  When the
approximating martingale does not have the predictable representation
property, an additional martingale orthogonal to the driving increments
appears naturally; this phenomenon is already present in the approximation
framework of Briand, Delyon and M\'emin~\cite{BDM2002}.  Convergence of
random-walk-driven BS\(\Delta\)Es has subsequently been studied in various
forms, including the general weak-convergence and stability results of
Madan, Pistorius and Stadje~\cite{MadanPistoriusStadje2016}.  Quantitative
rates for random-walk approximations have also been obtained in several
settings, including the early simple-random-walk scheme of
Ma, Protter, San Mart\'in and Torres~\cite{MaProtterSanMartinTorres2002},
the Wasserstein-rate analysis of Briand, Geiss, Geiss and
Labart~\cite{BriandGeissGeissLabart2021}, and the \(L^2\)-rate results of
Geiss, Labart and Luoto~\cite{GeissLabartLuoto2020}.  The first-order result
below addresses a different question: in the smooth Markovian setting, it
identifies a moment condition on a general finite-state increment that yields
first-order value, gradient and predictable-field accuracy, and then relates
that condition to one-step predictable representation.  The preceding literature
highlights two questions that become coupled for a finite-state Markovian
scheme: how accurately the walk reproduces Brownian weak propagation, and
how much of the one-step martingale space is represented by stochastic
integration against the walk itself.  The central point of this paper is
that, for finite-state schemes using the same increment for both tasks,
these two requirements are not independent: the third-order tensor that
controls weak accuracy also constrains exact one-step predictable
representation.

This paper studies that coupling.  Let \(\xi\in\mathbb R^d\) be a
state-independent finite-state increment satisfying
\[
 \E[\xi]=0,\qquad
 \E[\xi\xi^\top]=I_d.
\]
The first two moments give the Brownian drift and covariance.  The next
moment controls the first non-Gaussian term in a smooth weak expansion.  If
\[
 \E[\xi^{\otimes3}]=0,
\]
then the order-\(h^{3/2}\) term disappears simultaneously from the weak
propagator and from the discrete gradient.  We show that, under smooth
Markovian assumptions, this local cancellation propagates through the
nonlinear implicit BS\(\Delta\)E and yields order-\(h\) convergence of the
value field, its spatial gradient, and the predictable \(Z\)-field.

The same condition has a representation-theoretic consequence.  Exact
scalar one-step predictable representation means that every centered
one-step scalar random variable is a linear function of \(\xi\).  In the
minimal finite-state case this is the discrete normal-martingale geometry
described by the structure equations of Attal and
\'Emery~\cite{AttalEmery1994} and by the theory of obtuse random
walks~\cite{FranzHamdi2015,AttalDeschampsPellegrini2016}.  The corresponding
structure coefficients are third-order moments.  As an elementary consequence
of this known structure-equation framework, we isolate the implication needed
for the approximation problem: in dimension
\(d\ge2\), exact scalar one-step predictable representation and complete
third-moment cancellation cannot coexist under centered isotropy.

Combining the two sides gives the main structural conclusion.  Every
finite-state centered isotropic exact-PRP increment in dimension \(d\ge2\)
has a direction \(a\) with
\[
 \E[(a^\top\xi)^3]\ne0.
\]
For the heat equation with a suitable smooth Fourier terminal condition,
the corresponding one-step multiplier has an order-\(h^{3/2}\) imaginary
defect.  Over \(T/h\) steps this produces a genuine order-\(h^{1/2}\) global
weak error.  Thus exact one-step representation rules out a uniform
first-order weak error bound for this class of single-increment schemes.  If
exact representation is relaxed, third-moment cancellation becomes
possible; the unrepresented part then appears as an orthogonal martingale,
whose accumulated scale is \(O(h^{1/2})\).  This orthogonal martingale
is the part of the one-step continuation-value innovation that cannot be
represented by stochastic integration against the driving walk.  Its
accumulated scale therefore quantifies a residual exposure outside the
walk-generated representation sector, not merely an auxiliary term in the
BS\(\Delta\)E decomposition.

A centrally symmetric regular simplex makes this tradeoff explicit.  The
minimal regular simplex has \(d+1\) states and exact scalar one-step PRP but,
for \(d\ge2\), a nonzero third moment.  Adding the opposite vertices gives a
\(2(d+1)\)-state isotropic law with all odd moments equal to zero.  The
symmetric rule is therefore first-order in the smooth BS\(\Delta\)E
problem, but its larger one-step \(L^2\) space contains a nontrivial GKW
complement.  For a Fourier benchmark we identify the leading quadratic
component in that complement and determine the exact limit of
\(h^{-1}\E[|N_N^h|^2]\), showing that it is strictly positive.
More generally, we identify the quadratic GKW representation defect that
governs this residual fluctuation.  Under third-moment cancellation we prove
a dimension-dependent universal lower bound for its total size on the
trace-free quadratic sector.  We then prove a joint functional limit theorem
for the state process and $N^h/\sqrt h$.  The limiting orthogonal component
is a stochastic integral driven by an independent Brownian innovation, with
path-dependent quadratic variation determined by the Hessian of the limiting
value field.

\subsection*{Contributions and relation to prior work}

The individual ingredients of the argument are classical in isolation.
Third-moment matching is a standard weak-approximation principle; see, for
example, Kloeden and Platen~\cite{KloedenPlaten1992} and Lyons and
Victoir~\cite{LyonsVictoir2004}.  Minimal finite-state
normal martingales and their third-order structure tensors belong to the
discrete martingale literature cited above.  Orthogonal martingales in BSDE
approximations without predictable representation are also classical.  The
contribution here is to identify the \emph{same} third-order tensor as the
quantity governing both mechanisms and to turn that observation into sharp
positive and negative approximation results.  The cubature-based BSDE schemes of Crisan and
Manolarakis~\cite{CrisanManolarakis2012,CrisanManolarakis2014} also include a
second-order construction; Runge--Kutta BSDE schemes provide
another route to higher-order discretization~\cite{ChassagneuxCrisan2014}.
Certain degree-three Brownian cubature rules may also be represented by
finite-support one-step increments satisfying the same low-order moment
conditions.  Thus neither finite support nor first-order accuracy alone is
claimed as novel here.  Our focus is the simultaneous analysis of numerical
accuracy and exact scalar one-step predictable representation for the
\emph{same} increment law; the obstruction does not apply to higher-order
BSDE algorithms in general.  Recombining
multinomial trees for multiple assets have also been studied by Sierag and
Hanzon~\cite{SieragHanzon2018}; the issue here is whether the third-moment condition and exact one-step scalar
PRP can hold simultaneously for a centered isotropic increment.  Classical Brownian time-discretization error expansions for \(Y,Z\) are
given by Gobet--Labart~\cite{GobetLabart2007}, and finite-state
BS\(\Delta\)Es have an independent discrete-time theory
\cite{CohenElliott2010}.  The moment conditions here are also degree-three
Gaussian cubature conditions; see Cools~\cite{Cools1997}.  Our point is the
specific coupling of those conditions with exact scalar one-step PRP.

More precisely, the paper proves the following.

\begin{enumerate}[leftmargin=2.2em]
\item If a state-independent finite-state increment is centered, isotropic,
and third-moment-free, then the smooth Markovian
BS\(\Delta\)E has first-order value, gradient, and predictable-\(Z\)
accuracy.  The orthogonal martingale satisfies
\[
 \E[\max_{n\le N}|N_n^h|^2]=O(h).
\]
Moreover, the quadratic GKW representation defect obeys the sharp universal
operator lower bound $d/[2(d+2)]$ on the trace-free symmetric sector, with
equality exactly for normalized laws that are weighted spherical $4$-designs.
After normalization by
\(h^{-1/2}\), the orthogonal martingale admits a joint functional limit
with the state process; the limiting orthogonal noise is conditionally
Gaussian, with quadratic variation determined by that defect.

\item For \(d\ge2\), exact scalar one-step PRP is incompatible with
\(\E[\xi^{\otimes3}]=0\).  Every centered isotropic finite-state
exact-PRP law therefore admits a bounded smooth Fourier terminal condition
for which the global weak error is \(\Theta(h^{1/2})\).

\item For the centrally symmetric simplex, the \(O(h^{1/2})\)
orthogonal-martingale scale is sharp.  We compute its exact leading variance
constant.  Exact Fourier calculations and a semilinear manufactured
solution provide numerical checks of the predicted rates.
\end{enumerate}

The result is deliberately restricted to a single state-independent
finite-state increment reused at each time step.  It does not rule out
first- or higher-order BSDE algorithms that enlarge the one-step
information set, use several stages, iterated integrals, regression
variables, cubature constructions, or auxiliary randomness.

Fukasawa, Sato and Sekine~\cite{FukasawaSatoSekine2026} have recently studied the
minimal \(d+1\)-state lattice BS\(\Delta\)E from a complementary structural
viewpoint.  Their results concern, among other things, a constraint on the
driver, measure changes, nonlinear expectations, dynamically complete
lattice markets, and equilibrium.  In the present paper the same
minimal-state geometry is used instead as the exact-representation
benchmark: the question is what weak-approximation cost is imposed by that
representation property.  In this sense, their work exploits minimality,
whereas the present work quantifies the approximation gain obtained by
relaxing it.  A companion manuscript~\cite{TanakaMinimal2026} studies
finite-step quadratic replication and lattice geometry.  Its quadratic-budget
trace identity is closely related to the defect identity underlying
Theorem~\ref{thm:universal-defect}, whereas its main focus is the minimally
complete endpoint and Gamma-to-Delta aliasing rather than the asymptotic GKW
residual studied here.

The minimal-simplex Markovian BS\(\Delta\)E and its rough-data Wasserstein
convergence analysis appeared earlier in Chapter~8 of the author's doctoral
thesis~\cite{Thesis}.  Under Lipschitz/H\"older assumptions, that chapter
obtains \(O((1+\log N)N^{-1/2})\)-type bounds by a Wasserstein
central-limit comparison.  The present smooth first-order theorem uses a
different consistency--stability argument.  New results of the present paper
include the quadratic GKW representation defect and its sharp universal
trace-free lower bound, the joint functional limit of the state process and
normalized orthogonal martingale, the third-moment/representation
obstruction, the sharp Fourier lower bound, and the symmetric-simplex
orthogonal-martingale asymptotic.

\section{Third-moment-free increment laws and simplex examples}

We first state the moment class that drives the convergence theorem.

\begin{assumption}[Increment law]\label{ass:increment}
The random vector \(\xi\in\mathbb R^d\) has finite support, its law is
independent of the state and time level, and
\begin{equation}\label{eq:general-moments}
 \E[\xi]=0,\qquad
 \E[\xi\xi^\top]=I_d,\qquad
 \E[\xi^{\otimes3}]=0.
\end{equation}
Write
\[
 \kappa_4:=\E[|\xi|^4]<\infty.
\]
\end{assumption}

\begin{remark}[Why finite support is retained]\label{rem:finitesupport}
The local consistency and stability estimates below use only the displayed
moments (in particular \(\kappa_4<\infty\)) together with a
state-independent increment law; finite support is not essential to that
analytic part.  We retain it in the theorem because the second half of the
paper concerns the dimension of the one-step \(L^2\) space, exact
predictable representation, recombining rules, and finite-state
completeness.  Thus the convergence theorem is stated on the common class
on which the accuracy--representation comparison is made.
\end{remark}

For a smooth scalar function $\phi$, the notation
$\|D^k\phi\|_\infty$ means the supremum of the operator norm
of its $k$-linear derivative.  In particular, Taylor remainders are
estimated using $|D^k\phi(x)[v,\ldots,v]|\le
\|D^k\phi\|_\infty |v|^k$.

For a bounded function \(\phi\), define
\begin{align}
 P_h\phi(x)&:=\E[\phi(x+\sqrt h\,\xi)],\\
 D_h\phi(x)&:=\frac1{\sqrt h}\E[\phi(x+\sqrt h\,\xi)\xi],\\
 L_h\phi(x)&:=\frac{P_h\phi(x)-\phi(x)}h.
\end{align}

\begin{proposition}[Third-moment-free local consistency]\label{prop:local}
Under Assumption~\ref{ass:increment}, for \(\phi\in C_b^4(\R^d)\),
\[
 \norm{L_h\phi-\tfrac12\Delta\phi}_\infty
 \le \frac{\kappa_4}{24}h\norm{D^4\phi}_\infty,
\]
and for \(\phi\in C_b^3(\R^d)\),
\[
 \norm{D_h\phi-\nabla\phi}_\infty
 \le \frac{\kappa_4}{6}h\norm{D^3\phi}_\infty.
\]
\end{proposition}

\begin{proof}
Taylor expansion gives
\[
 \phi(x+\sqrt h\,\xi)
 =\phi(x)+\sqrt h\,D\phi[\xi]
 +\frac h2D^2\phi[\xi,\xi]
 +\frac{h^{3/2}}6D^3\phi[\xi,\xi,\xi]+R_4,
\]
with
\[
 |R_4|\le \frac{h^2}{24}\norm{D^4\phi}_\infty|\xi|^4.
\]
Centering, isotropy and the vanishing third moment give the first estimate.
For \(D_h\), Taylor expansion through second order gives the leading term
\(\nabla\phi\); the order-\(\sqrt h\) contribution is the contraction of
\(D^2\phi\) with \(\E[\xi^{\otimes3}]\), hence vanishes, and the vector
remainder is bounded by
\[
 \frac h6\norm{D^3\phi}_\infty\E[|\xi|^4].
\]
\end{proof}

\begin{remark}[Non-identity covariance]
For positive definite covariance \(\Sigma\), whitening by
\(\Sigma^{-1/2}\) preserves exact scalar PRP and vanishing third moment.
Thus the representation obstruction is unchanged by a nondegenerate linear
covariance transformation.
\end{remark}

\subsection{Regular-simplex examples}

Let \(\widehat v_0,\ldots,\widehat v_d\in\R^d\) satisfy
\begin{equation}\label{eq:simplex}
 \sum_{i=0}^d\widehat v_i=0,\qquad
 \frac1{d+1}\sum_{i=0}^d\widehat v_i\widehat v_i^\top=I_d,
\end{equation}
or equivalently
\[
 |\widehat v_i|^2=d,\qquad
 \widehat v_i^\top\widehat v_j=-1\quad(i\ne j).
\]
Such a family is unique up to an orthogonal transformation.

The minimal simplex increment \(\xi^{\rm min}\) is uniform on
\(\{\widehat v_i\}_{i=0}^d\).  The symmetric simplex increment \(\xi\) is
defined by
\begin{equation}\label{eq:sym-law}
 \mathbb P(\xi=\widehat v_i)
 =\mathbb P(\xi=-\widehat v_i)
 =\frac1{2(d+1)}.
\end{equation}
Both satisfy
\[
 \E[\xi]=0,\qquad \E[\xi\xi^\top]=I_d,\qquad |\xi|^2=d,
\]
while the symmetric law additionally satisfies
\begin{equation}\label{eq:odd}
 \E[\xi^{\otimes(2k+1)}]=0,\qquad k\ge0.
\end{equation}

For comparison, the minimal simplex generally retains the third-moment
terms
\[
 L_h^{\rm min}\phi
 =\frac12\Delta\phi
 +\frac{\sqrt h}{6}\E[D^3\phi[\xi^{\rm min},\xi^{\rm min},\xi^{\rm min}]]
 +O(h),
\]
and
\[
 D_h^{\rm min}\phi
 =\nabla\phi
 +\frac{\sqrt h}{2}
 \E\!\left[\xi^{\rm min}
 D^2\phi[\xi^{\rm min},\xi^{\rm min}]\right]+O(h).
\]

\begin{remark}[A second first-order finite-state rule]\label{rem:axisrule}
The convergence theorem is not specific to the symmetric simplex.  For
\(d\ge1\), the axis rule
\[
 \mathbb P(\xi=\sqrt d\,e_j)
 =\mathbb P(\xi=-\sqrt d\,e_j)=\frac1{2d},
 \qquad j=1,\ldots,d,
\]
also satisfies
\[
 \E[\xi]=0,\qquad \E[\xi\xi^\top]=I_d,\qquad
 \E[\xi^{\otimes3}]=0,\qquad \E[|\xi|^4]=d^2,
\]
so Theorem~\ref{thm:main} applies.  For \(d\ge2\) it has
\(2d\) one-step atoms, so the centered scalar space has dimension \(2d-1\)
and the GKW complement has dimension \(d-1\).  We use the symmetric simplex
for the sharp comparison because it is the direct antipodal extension of
the minimal regular-simplex rule and keeps the Fourier geometry explicit.
\end{remark}

\begin{proposition}[Minimal support and one-step PRP]\label{prop:minprp}
Let a centered nondegenerate increment in \(\R^d\) have \(m\) distinct atoms
of positive probability.  The centered scalar one-step space has dimension
\(m-1\), whereas the linear stochastic-integral sector
\(\{a^\top\xi:a\in\R^d\}\) has dimension \(d\).  Hence exact scalar
one-step PRP is possible only if \(m=d+1\), and for a nondegenerate
\(d+1\)-point isotropic law it holds automatically.
\end{proposition}
\begin{proof}
Scalar functions on \(m\) atoms form an \(m\)-dimensional space and the
mean-zero hyperplane has dimension \(m-1\).  Isotropy makes the coordinate
functions linearly independent in \(L^2\), so their span has dimension
\(d\).  Moreover, a centered nondegenerate \(d\)-dimensional law must have at least \(d+1\) affinely spanning support points, and hence \(m\ge d+1\).  Exact scalar one-step PRP requires the \(m-1\)-dimensional centered scalar space to coincide with the \(d\)-dimensional linear increment sector, and therefore \(m-1=d\).
\end{proof}

\begin{remark}
For the symmetric simplex, \(m=2(d+1)\), so the centered scalar space has
dimension \(2d+1\).  The linear increment sector still has dimension \(d\);
the GKW complement therefore has dimension \(d+1\) when \(d\ge2\).
The one-dimensional case is exceptional because the labels
\(\pm\widehat v_i\) duplicate the same two Bernoulli support points.
\end{remark}

\section{The finite-state BS\texorpdfstring{\(\Delta\)}{}E}

Let \(t_n=nh\), \(h=T/N\), and
\[
 B_{n+1}^h=B_n^h+\sqrt h\,\xi_{n+1},
\]
where the increments are i.i.d.\ with a law satisfying
Assumption~\ref{ass:increment}.  We consider the Markov recursion
\begin{equation}\label{eq:scheme}
 \begin{split}
 U_N(x)&=g(x),\\
 U_n(x)&=P_hU_{n+1}(x)
 +h f\!\left(t_{n+1},x,U_n(x),D_hU_{n+1}(x)\right).
 \end{split}
\end{equation}
Set
\[
 Y_n^h=U_n(B_n^h),\qquad
 Z_{n+1}^h=D_hU_{n+1}(B_n^h).
\]

For every \(d\ge1\), the one-step Galtchouk--Kunita--Watanabe
decomposition with respect to the linear span of the increment takes the
form
\begin{equation}\label{eq:gkw}
 Y_{n+1}^h-\E_n[Y_{n+1}^h]
 =(Z_{n+1}^h)^\top\Delta B_{n+1}^h+\Delta N_{n+1}^h,
\end{equation}
where \(N_0^h=0\) and
\[
 \E_n[\Delta N_{n+1}^h]=0,\qquad
 \E_n[\Delta N_{n+1}^h\,\Delta B_{n+1}^h]=0.
\]
Thus \(N^h\) records the part of the one-step martingale increment that is
not represented by the linear increment sector; it may vanish identically.
In particular, in the one-dimensional Bernoulli case the centered one-step
scalar \(L^2\)-space is already exhausted by the linear increment, so that
\(N^h\equiv0\).  By contrast, for \(d\ge2\), Theorem~\ref{thm:prp} below
shows that a law satisfying Assumption~\ref{ass:increment} cannot have exact
scalar one-step PRP.  The multidimensional representation obstruction is
therefore separate from the validity of the decomposition~\eqref{eq:gkw}
and from the convergence theorem below, both of which also cover \(d=1\).

\section{Assumptions and the main convergence theorem}

Let \(u\) be a classical solution of
\begin{equation}\label{eq:pde}
 \partial_tu+\frac12\Delta u+
 f(t,x,u,\nabla u)=0,\qquad u(T,x)=g(x).
\end{equation}

\begin{assumption}[Smooth numerical regime]\label{ass:Hnum}
The following conditions hold.

\begin{enumerate}[label=(H\arabic*),leftmargin=2.5em]
\item The driver \(f:[0,T]\times\R^d\times\R\times\R^d\to\R\)
is continuous and globally Lipschitz in \((x,y,z)\), uniformly in time.
Let \(L_y\) and \(L_z\) denote Lipschitz constants in the \(y\)- and
\(z\)-variables, respectively.  It is Lipschitz in \(t\), uniformly in
\(x,y,z\).

\item The map \(f\) is continuously differentiable in \((x,y,z)\).
The derivatives \(f_x,f_y,f_z\) are globally bounded and jointly Lipschitz
in \((t,x,y,z)\).  In addition,
\[
 f_y(t,x,y,z)\le \mu_y
\]
on \( [0,T]\times\R^d\times\R\times\R^d \).  Equivalently, the \(y\)-dependence is one-sided Lipschitz with
constant \(\mu_y\).

\item The semilinear PDE
\[
 \partial_tu+\frac12\Delta u+f(t,x,u,\nabla u)=0,\qquad u(T,\cdot)=g,
\]
has a classical solution on the closed strip.  All lower derivatives used
below are bounded, and in particular
\[
 u_{tt},\quad \nabla u_{tt},\quad
 D^5u,\quad D^4u,\quad D^3u,\quad
 D^3u_t,\quad D^2u_t,\quad \nabla u_t
\]
are bounded.

\item The mesh is sufficiently fine that
\[
 \delta_h:=1-h\mu_y>0.
\]
\end{enumerate}
\end{assumption}

\begin{remark}[On the regularity package]\label{rem:regularity}
Assumption~\ref{ass:Hnum} is intentionally stated at the solution level.
Its purpose is to make the numerical argument independent of a particular
high-order semilinear parabolic regularity theorem.  The proof uses no
sixth spatial derivative of \(u\) and no second time derivative of \(f\).
In particular, the bounds on \(u_{tt}\) and \(\nabla u_{tt}\) are
assumed at the solution level and used in the time-Taylor remainders in
Appendix~A; they are not derived here from the first-order regularity of
the driver.
Under standard smooth bounded coefficient hypotheses, the displayed
solution regularity can instead be obtained from classical parabolic/BSDE
regularity results.  We keep that analytic input separate from the
finite-state approximation argument.
\end{remark}

\begin{theorem}[First-order approximation under third-moment cancellation]\label{thm:main}
Under Assumptions~\ref{ass:increment} and \ref{ass:Hnum}, there exists a constant \(C\), independent
of sufficiently small \(h=T/N\), such that
\begin{align}
 \max_{0\le n\le N}\norm{U_n-u(t_n,\cdot)}_\infty&\le Ch,\label{eq:value}\\
 \max_{0\le n\le N}\norm{\nabla U_n-\nabla u(t_n,\cdot)}_\infty&\le Ch,\label{eq:grad}\\
 \max_{0\le n<N}\norm{D_hU_{n+1}-\nabla u(t_{n+1},\cdot)}_\infty&\le Ch.
 \label{eq:zfield}
\end{align}
Moreover, for the orthogonal martingale in \eqref{eq:gkw},
\begin{equation}\label{eq:N}
 \E[\max_{0\le n\le N}|N_n^h|^2]\le Ch.
\end{equation}
\end{theorem}

\begin{remark}[Time index of the predictable field]\label{rem:Zindex}
With the backward convention used here,
\[
 Z^h_{n+1}=D_hU_{n+1}(B_n^h)
 =\frac1h\E_n[Y^h_{n+1}\Delta B^h_{n+1}],
\]
so the deterministic field \(D_hU_{n+1}\) is naturally compared with
\(\nabla u(t_{n+1},\cdot)\).  This is the same indexing convention as in the
earlier lattice BS\(\Delta\)E.
\end{remark}

\begin{corollary}[Symmetric simplex]\label{cor:symmain}
The centrally symmetric simplex law \eqref{eq:sym-law} satisfies
Assumption~\ref{ass:increment} with \(\kappa_4=d^2\).  Hence
Theorem~\ref{thm:main} applies to the symmetric-simplex BS\(\Delta\)E.
\end{corollary}

\section{Proof of the convergence theorem}

The proof is arranged in dependency order.  In particular, the value
estimate is established without using the gradient or martingale estimates;
the gradient estimate uses the value estimate; and the maximal
orthogonal-martingale estimate of this section uses only the value estimate.
The refined local GKW expansion in Section~\ref{sec:fclt}, by contrast,
also uses the gradient estimate.

\subsection{Well-posedness of the implicit step}

Fix \(n\) and suppose \(U_{n+1}\) is known.  At a fixed \(x\), put
\[
 r=P_hU_{n+1}(x),\qquad z=D_hU_{n+1}(x),
\]
and define
\[
 G(y)=y-hf(t_{n+1},x,y,z).
\]
For \(y,\widetilde y\in\R\), the one-sided Lipschitz condition gives
\[
 (G(y)-G(\widetilde y))(y-\widetilde y)
 \ge (1-h\mu_y)|y-\widetilde y|^2.
\]
Thus \(G\) is strictly increasing.  Taking \(\widetilde y=0\) also gives
linear coercivity as \(|y|\to\infty\), and continuity implies that \(G\) is
onto.  Therefore the scalar implicit equation has a unique solution.
Backward induction defines \(U_n\) uniquely.  Notice that no Banach
contraction condition \(hL_y<1\) is needed.

\subsection{Value consistency}

Write \(u_n=u(t_n,\cdot)\) and define the exact residual
\[
 \tau_n^u
 =u_n-P_hu_{n+1}
 -h f(t_{n+1},x,u_n,D_hu_{n+1}).
\]
By Proposition~\ref{prop:local} and Taylor expansion,
\[
 P_hu_{n+1}
 =u_{n+1}+\frac h2\Delta u_{n+1}+O(h^2)
 =u_n+h(u_t+\tfrac12\Delta u)_n+O(h^2),
\]
uniformly in \(x\).  Furthermore,
\[
 D_hu_{n+1}-\nabla u_n
 =(D_hu_{n+1}-\nabla u_{n+1})
 +(\nabla u_{n+1}-\nabla u_n)
 =O(h).
\]
The first term follows from Proposition~\ref{prop:local}; the second from
bounded \(\nabla u_t\).  Using the PDE at \((t_n,x)\), and the Lipschitz
dependence of \(f\) in \(t\) and \(z\), yields
\begin{equation}\label{eq:tauu-detailed}
 \norm{\tau_n^u}_\infty\le Ch^2.
\end{equation}
This is a pure consistency estimate and contains no numerical error.

\subsection{Value stability without positive tilted weights}

Let
\[
 e_n=u_n-U_n.
\]
For two pairs \((y,z)\), \((\widetilde y,\widetilde z)\), let \(a\)
denote the scalar secant in the \(y\)-direction,
\[
 a=
 \begin{cases}
 \displaystyle
 \frac{f(t,x,y,z)-f(t,x,\widetilde y,z)}{y-\widetilde y},
 &y\ne\widetilde y,\\[1.2ex]
 f_y(t,x,y,z),&y=\widetilde y,
 \end{cases}
\]
and let \(b\) denote the secant vector in the \(z\)-direction, defined,
when \(z\ne\widetilde z\), by
\[
 b=
 \frac{f(t,x,\widetilde y,z)-f(t,x,\widetilde y,\widetilde z)}
 {|z-\widetilde z|^2}(z-\widetilde z),
\]
with \(b=0\) otherwise.  Then
\[
 f(t,x,y,z)-f(t,x,\widetilde y,\widetilde z)
 =a(y-\widetilde y)+b^\top(z-\widetilde z),
\]
where \(a\le\mu_y\) and \(|b|\le L_z\).
At each \(x\), apply these secants with
\[
 t=t_{n+1},\qquad
 (y,z)=\bigl(u_n(x),D_hu_{n+1}(x)\bigr),\qquad
 (\widetilde y,\widetilde z)
 =\bigl(U_n(x),D_hU_{n+1}(x)\bigr),
\]
and denote the resulting coefficients by \(a_n(x)\) and \(b_n(x)\).
Subtracting the numerical and exact recursions then gives
\begin{equation}\label{eq:value-rec-detailed}
 (1-ha_n)e_n=T_{b_n}e_{n+1}+\tau_n^u,
\end{equation}
where, for a bounded vector field \(b\),
\[
 T_b\phi(x)
 =\E\!\left[\phi(x+\sqrt h\,\xi)
 (1+\sqrt h\,b(x)^\top\xi)\right].
\]
The weights in this expression need not be positive.  Cauchy--Schwarz and
isotropy give
\[
 |T_b\phi(x)|
 \le \norm{\phi}_\infty
 \left(\E[(1+\sqrt h\,b(x)^\top\xi)^2]\right)^{1/2}
 =\sqrt{1+h|b(x)|^2}\,\norm{\phi}_\infty.
\]
Hence, with \(\delta_h=1-h\mu_y\),
\[
 \norm{e_n}_\infty
 \le \rho_h\norm{e_{n+1}}_\infty
 +C\delta_h^{-1}h^2,\qquad
 \rho_h=\frac{\sqrt{1+hL_z^2}}{1-h\mu_y}.
\]
For sufficiently small \(h\), \(\rho_h=1+O(h)\) and
\(\sup_{k\le T/h}\rho_h^k<\infty\).  Since \(e_N=0\),
\[
 \max_n\norm{e_n}_\infty\le Ch.
\]
This proves \eqref{eq:value}.

\subsection{Differentiability and gradient consistency}

Assume inductively that \(U_{n+1}\in C^1\).  Since
\(1-hf_y\ge\delta_h>0\), the scalar implicit-function theorem gives
\(U_n\in C^1\).  Thus \(C^1\) regularity propagates backward from
\(U_N=g\).

Set \(P_n=\nabla U_n\) and
\[
 \Theta_n^{\rm num}
 =(t_{n+1},x,U_n(x),D_hU_{n+1}(x)),
\]
and write \(f_\bullet^{{\rm num},n}=f_\bullet(\Theta_n^{\rm num})\) for
\(\bullet\in\{x,y,z\}\).  For a vector field \(P:\R^d\to\R^d\), define
\[
 (J_hP)_{jk}(x):=h^{-1/2}\E[P_j(x+\sqrt h\,\xi)\xi_k].
\]
With column gradients, \(J_h(\nabla U)=\nabla(D_hU)\).  Differentiating the
recursion therefore gives
\begin{equation}\label{eq:P-rec}
 P_n=P_hP_{n+1}
 +h\{f_x^{{\rm num},n}+f_y^{{\rm num},n}P_n
 +(J_hP_{n+1})f_z^{{\rm num},n}\}.
\end{equation}
Hence \(h(J_hq)b=\sqrt h\,\E[q(x+\sqrt h\,\xi)\xi^\top b]\), and the
Cauchy--Schwarz argument for \(T_b\) gives
\(\|T_bq\|_\infty\le\sqrt{1+h\|b\|_\infty^2}\|q\|_\infty\) also for vector fields.
For \(p_n=\nabla u_n\), set
\[
 \Theta_n^\ast
 =(t_{n+1},x,u_n(x),D_hu_{n+1}(x)),
\]
and write \(f_\bullet^{*,n}=f_\bullet(\Theta_n^\ast)\) for
\(\bullet\in\{x,y,z\}\).  Differentiation of the PDE consistency expansion
gives the analogous identity
\[
 p_n=P_hp_{n+1}
 +h\{f_x^{*,n}+f_y^{*,n}p_n+(J_hp_{n+1})f_z^{*,n}\}
 +\tau_n^p.
\]
The derivative package in Assumption~\ref{ass:Hnum} gives
\begin{equation}\label{eq:taup-detailed}
 \norm{\tau_n^p}_\infty\le Ch^2.
\end{equation}
A term-by-term derivation is recorded in Appendix~A.

\subsection{Gradient stability and absence of an inverse square-root loss}

Let
\[
 q_n=p_n-P_n=\nabla u_n-\nabla U_n.
\]
Subtracting the differentiated recursions, define the numerical
derivative coefficients
\[
 \alpha_n(x):=f_y^{{\rm num},n}(x),\qquad
 \beta_n(x):=f_z^{{\rm num},n}(x).
\]
Then
\[
 (1-h\alpha_n)q_n=T_{\beta_n}q_{n+1}+h\eta_n+\tau_n^p.
\]
Here
\[
 \eta_n
 =(f_x^{*,n}-f_x^{{\rm num},n})
 +(f_y^{*,n}-f_y^{{\rm num},n})p_n
 +(J_hp_{n+1})(f_z^{*,n}-f_z^{{\rm num},n}).
\]
Thus coefficient differences multiply exact-solution quantities only.
Since \(\alpha_n\le\mu_y\), \(|\beta_n|\le L_z\), and the derivatives of
\(f\) are jointly Lipschitz,
\[
 \norm{\eta_n}_\infty
 \le C_0\left(\norm{e_n}_\infty+
            \norm{D_he_{n+1}}_\infty\right),
\]
where \(C_0\) depends only on the constants in Assumption~\ref{ass:Hnum}
and on the corresponding uniform bounds for the exact solution \(u\); in
particular, it is independent of \(h\), \(n\), and of any a priori bound on
\(P_k=\nabla U_k\).

The potentially problematic term \(D_he\) does not cost \(h^{-1/2}\).
Indeed, since \(\E[\xi]=0\),
\[
 D_he(x)=
 \E\!\left[
 \frac{e(x+\sqrt h\,\xi)-e(x)}{\sqrt h}\,\xi
 \right].
\]
The fundamental theorem of calculus, followed by expectation, yields
\[
 |D_he(x)|\le \norm{\nabla e}_\infty\E[|\xi|^2]
 =d\norm{\nabla e}_\infty,
\]
and hence
\begin{equation}\label{eq:no-loss}
 \norm{D_he}_\infty\le d\norm{\nabla e}_\infty.
\end{equation}
Using the already proved value estimate and the identity
\(\nabla e_{n+1}=q_{n+1}\),
\[
 \norm{\eta_n}_\infty
 \le C_0(h+\norm{q_{n+1}}_\infty).
\]
This is a closed backward estimate: it involves only the already established
value error and the next-step gradient error.  Consequently, the tilted-operator estimate gives, with a constant \(C\) depending only on the same fixed data as \(C_0\),
\[
 \norm{q_n}_\infty
 \le(1+Ch)\norm{q_{n+1}}_\infty+Ch^2.
\]
Since \(q_N=0\), backward discrete Gronwall proves
\[
 \max_n\norm{q_n}_\infty\le Ch,
\]
which is \eqref{eq:grad}.

Finally,
\[
 D_hU_{n+1}-\nabla u_{n+1}
 =-D_he_{n+1}+
 (D_hu_{n+1}-\nabla u_{n+1}).
\]
The first term is \(O(h)\) by \eqref{eq:no-loss} and the gradient estimate,
and the second is \(O(h)\) by Proposition~\ref{prop:local}.  This proves
\eqref{eq:zfield}.

\subsection{Orthogonal martingale}

Conditionally on \(\mathcal F_n\), let \(\Pi_\perp\) denote the orthogonal
projection in the finite one-step \(L^2\)-space onto the complement of
constants and the linear span of \(\xi\).  Then
\[
 \Delta N_{n+1}^h
 =\Pi_\perp[
 U_{n+1}(B_n^h+\sqrt h\,\xi)].
\]
At \(x=B_n^h\),
\[
 u_{n+1}(x+\sqrt h\,\xi)
 =u_{n+1}(x)+\sqrt h\,\nabla u_{n+1}(x)^\top\xi+R_{n+1},
\]
where
\[
 |R_{n+1}|\le \frac h2\norm{D^2u}_\infty|\xi|^2.
\]
Projection kills the constant and linear terms.  Since orthogonal projection
is an \(L^2\) contraction, \(\E[|\xi|^4]=\kappa_4<\infty\), and
\(\norm{U_{n+1}-u_{n+1}}_\infty=O(h)\),
\[
 \E_n[|\Delta N_{n+1}^h|^2]\le Ch^2.
\]
The martingale increments are orthogonal across time, hence
\[
 \E[|N_N^h|^2]
 =\sum_{n=0}^{N-1}\E[|\Delta N_{n+1}^h|^2]
 \le Ch.
\]
Doob's \(L^2\) inequality gives
\[
 \E[\max_{0\le n\le N}|N_n^h|^2]\le Ch.
\]
This branch of the proof uses the value estimate but not the gradient or
\(Z\)-field estimate.

\subsection{Dependency graph}

The proof can be summarized by
\[
 \begin{array}{ccc}
 \text{local value consistency}&\longrightarrow&\text{value }O(h)\\
 &&\downarrow\\
 \text{local gradient consistency}&\longrightarrow&\text{gradient }O(h)
 \longrightarrow Z\text{-field }O(h),\\
 &\searrow&\\[-1ex]
 &&N^h:\ \E[\max|N^h|^2]=O(h).
 \end{array}
\]
The last branch records only the maximal \(L^2\) estimate proved above.
The refined expansion and functional limit in Section~\ref{sec:fclt} also
use the gradient estimate.  In particular there is no circular use of the
final \(Z\)-estimate in the gradient proof.

\section{Quadratic representation defect and functional limit}\label{sec:fclt}

The preceding maximal estimate identifies the scale of the orthogonal
martingale but not its limiting law.  We now refine the one-step GKW
expansion.  The result also identifies the fourth-order quantity that
controls the residual fluctuation once the third moment has been cancelled.

For $A\in\mathbb S^d$, put
\[
 \mathcal H_1:=\operatorname{span}\{1,\xi_1,\ldots,\xi_d\}\subset L^2(\mu),
 \qquad q_A(\xi):=\frac12\xi^\top A\xi,
\]
and define
\[
 \mathcal R_A:=(I-P_{\mathcal H_1})q_A,
 \qquad
 \Gamma_\mu(A):=\E[\mathcal R_A(\xi)^2].
\]
Writing
\[
 M_{ijkl}:=\E[\xi_i\xi_j\xi_k\xi_l],\qquad
 M_4[A,A]:=\sum_{i,j,k,l}M_{ijkl}A_{ij}A_{kl},
\]
we first record the one-step geometry.

\begin{proposition}[Quadratic representation defect]\label{prop:quad-defect}
Under Assumption~\ref{ass:increment}, for every $A\in\mathbb S^d$,
\[
 \mathcal R_A(\xi)=\frac12\{\xi^\top A\xi-\operatorname{tr}A\},
\]
and
\begin{equation}\label{eq:Gamma-defect}
 \Gamma_\mu(A)=\frac14\{M_4[A,A]-(\operatorname{tr}A)^2\}.
\end{equation}
\end{proposition}
\begin{proof}
By centering and isotropy, $1,\xi_1,\ldots,\xi_d$ are orthonormal in
$L^2(\mu)$.  Third-moment cancellation gives $\E[q_A\xi_k]=0$ for every
$k$, while $\E[q_A]=\frac12\operatorname{tr}A$.  Hence the displayed
residual formula follows from orthogonal projection.  Taking its squared
$L^2$ norm gives \eqref{eq:Gamma-defect}.
\end{proof}

\begin{remark}[General centered isotropic law]\label{rem:general-defect}
Without third-moment cancellation, if
$T_{ijk}=\E[\xi_i\xi_j\xi_k]$ and
$(T:A)_k=\sum_{i,j}T_{ijk}A_{ij}$, the same projection calculation gives
\[
 \mathcal R_A(\xi)
 =\frac12\{\xi^\top A\xi-\operatorname{tr}A-(T:A)^\top\xi\},
\]
\[
 \Gamma_\mu(A)
 =\frac14\{M_4[A,A]-(\operatorname{tr}A)^2-|T:A|^2\}.
\]
We use only the third-moment-free specialization in what follows.  This
restriction is also essential for the functional limit below: without
third-moment cancellation the first-order gradient estimate used here
is no longer available.  An $O(h^{1/2})$ gradient error may remain, in
which case the numerical-error contribution in the refined GKW expansion
is of order $h$ rather than $h^{3/2}$ and need not separate from the
quadratic residual $\mathcal R_A$.
\end{remark}

\begin{theorem}[Universal quadratic-defect lower bound]\label{thm:universal-defect}
Under Assumption~\ref{ass:increment}, assume $d\ge2$ and let
\[
 \mathbb S_0^d:=\{A\in\mathbb S^d:\operatorname{tr}A=0\},
 \qquad D_0:=\dim\mathbb S_0^d=\frac{(d-1)(d+2)}2.
\]
For every Frobenius-orthonormal basis
$\{A_1,\ldots,A_{D_0}\}$ of $\mathbb S_0^d$,
\begin{equation}\label{eq:total-defect-lb}
 \sum_{k=1}^{D_0}\Gamma_\mu(A_k)
 =\frac14\left(1-\frac1d\right)\E[|\xi|^4]
 \ge \frac{d(d-1)}4.
\end{equation}
Consequently,
\begin{equation}\label{eq:worst-defect-lb}
 \sup_{\substack{A\in\mathbb S_0^d\\ |A|_F=1}}
 \Gamma_\mu(A)\ge \frac{d}{2(d+2)}.
\end{equation}
The lower bound in \eqref{eq:total-defect-lb} is sharp, and equality holds
if and only if $|\xi|^2=d$ almost surely.
\end{theorem}
\begin{proof}
For $A_k\in\mathbb S_0^d$, Proposition~\ref{prop:quad-defect} gives
$4\Gamma_\mu(A_k)=\E[(\xi^\top A_k\xi)^2]$.  If $P_0$ denotes the
Frobenius-orthogonal projection from $\mathbb S^d$ onto $\mathbb S_0^d$,
Parseval's identity yields
\[
 4\sum_{k=1}^{D_0}\Gamma_\mu(A_k)
 =\E[\|P_0(\xi\xi^\top)\|_F^2]
 =\E\!\left[\left\|\xi\xi^\top-\frac{|\xi|^2}{d}I_d\right\|_F^2\right]
 =\left(1-\frac1d\right)\E[|\xi|^4].
\]
Since $\E[|\xi|^2]=d$, Jensen's inequality gives
$\E[|\xi|^4]\ge d^2$, with equality if and only if $|\xi|^2=d$ almost
surely.  Finally, the supremum in \eqref{eq:worst-defect-lb} is at least
the average of $\Gamma_\mu(A_k)$ over the $D_0$ basis elements, which gives
$d/[2(d+2)]$.
\end{proof}

\begin{corollary}[Sharpness and extremizers]\label{cor:defect-extremizers}
Among finite laws satisfying Assumption~\ref{ass:increment}, the operator
bound \eqref{eq:worst-defect-lb} is sharp in every dimension $d\ge2$.
Equality holds if and only if the normalized finite law $u=\xi/\sqrt d$ is
a weighted spherical $4$-design on $S^{d-1}$.  Here a weighted spherical
$t$-design means a finite subset of $S^{d-1}$ with positive weights summing
to one whose weighted averages of all polynomials of degree at most $t$ agree
with their averages under normalized surface measure.
\end{corollary}
\begin{proof}
If $u$ is a weighted spherical $4$-design, its moments through degree four
agree with spherical measure.  Hence Assumption~\ref{ass:increment} holds
for $\xi=\sqrt d\,u$, and for $A\in\mathbb S_0^d$,
\[
 \E[(u^\top Au)^2]=\frac{2}{d(d+2)}|A|_F^2,
 \qquad
 \Gamma_\mu(A)=\frac{d}{2(d+2)}|A|_F^2.
\]
Thus equality is attained.  Such finite designs exist in every dimension;
see Seymour and Zaslavsky~\cite{SeymourZaslavsky1984}.

Conversely, suppose equality holds in \eqref{eq:worst-defect-lb}.  Since the
sum in \eqref{eq:total-defect-lb} is at least
$D_0 d/[2(d+2)]$ while every eigenvalue of the positive quadratic form
$\Gamma_\mu|_{\mathbb S_0^d}$ is at most its supremum, equality holds
throughout.  Therefore $|\xi|^2=d$ almost surely and
\[
 \Gamma_\mu(A)=\frac{d}{2(d+2)}|A|_F^2,
 \qquad A\in\mathbb S_0^d.
\]
With $u=\xi/\sqrt d$, polarization gives, for $A,B\in\mathbb S_0^d$,
\[
 \E[(u^\top Au)(u^\top Bu)]
 =\frac{2}{d(d+2)}\langle A,B\rangle_F.
\]
Writing an arbitrary symmetric matrix as
$A=A_0+(\operatorname{tr}A/d)I_d$ and using $|u|=1$ and isotropy then yields
\[
 \E[(u^\top Au)(u^\top Bu)]
 =\frac{2\langle A,B\rangle_F+
 (\operatorname{tr}A)(\operatorname{tr}B)}{d(d+2)}.
\]
This is the spherical fourth-moment identity.  The bilinear form
$(A,B)\mapsto \E[(u^\top Au)(u^\top Bu)]$ on $\mathbb S^d$ determines the
fully symmetric fourth-moment tensor of $u$.  Together with the degree one,
two and three moment identities in Assumption~\ref{ass:increment}, the
displayed identity therefore shows that all polynomial moments through
degree four agree with spherical measure.  Hence $u$ is a weighted spherical
$4$-design.
\end{proof}

\begin{remark}[Total versus directional defect]\label{rem:total-directional-defect}
Constant-radius rules $|\xi|^2=d$ minimize the total defect in
\eqref{eq:total-defect-lb}, but operator optimality additionally requires
isotropy of the defect over $\mathbb S_0^d$.  In $d=2$ the symmetric simplex
is the regular hexagon and attains \eqref{eq:worst-defect-lb}.  For $d>2$ it
does not: with $A_0=\widehat v_0\widehat v_0^\top-I_d$,
$|A_0|_F^2=d(d-1)$ and \eqref{eq:Gamma-rank-one} gives
\[
 \Gamma_{\rm sym}(A_0/|A_0|_F)=\frac{d-1}{4}
 >\frac{d}{2(d+2)}.
\]
\end{remark}

\begin{lemma}[Refined local GKW expansion]\label{lem:local-gkw-fclt}
Uniformly in $n$ and $x$,
\[
 \left\|\Pi_\perp U_{n+1}(x+\sqrt h\,\xi)
 -h\mathcal R_{D^2u(t_{n+1},x)}(\xi)\right\|_{L^2(\mu)}
 \le Ch^{3/2}.
\]
Consequently, with $A_n=D^2u(t_{n+1},B_n^h)$,
\[
 \Delta N_{n+1}^h=h\mathcal R_{A_n}(\xi_{n+1})+r_{n+1}^h,
 \qquad |r_{n+1}^h|\le Ch^{3/2}\quad\text{a.s.},
 \qquad \E_n[|r_{n+1}^h|^2]\le Ch^3.
\]
\end{lemma}
\begin{proof}
Write $e_{n+1}=U_{n+1}-u(t_{n+1},\cdot)$.  Taylor expansion of the exact
solution gives
\[
 u(t_{n+1},x+\sqrt h\xi)
 =u+\sqrt h\nabla u^\top\xi+\frac h2\xi^\top D^2u\xi+R,
 \qquad \|R\|_{L^2(\mu)}\le Ch^{3/2}.
\]
Projection kills the constant and linear terms, so the exact-solution part
is $h\mathcal R_{D^2u(t_{n+1},x)}+O_{L^2}(h^{3/2})$.
For the numerical error, subtract the constant and linear terms annihilated
by $\Pi_\perp$ and write
\begin{align*}
 e_{n+1}(x+\sqrt h\xi)
 &=e_{n+1}(x)+\sqrt h\nabla e_{n+1}(x)^\top\xi\\
 &\quad+\sqrt h\int_0^1
 [\nabla e_{n+1}(x+s\sqrt h\xi)-\nabla e_{n+1}(x)]^\top\xi\,ds.
\end{align*}
Since the support of $\xi$ is finite, the Taylor remainder above is
uniformly $O(h^{3/2})$ on that support.  The same finite-support bound,
together with
\[
 \left|\sqrt h\int_0^1
 [\nabla e_{n+1}(x+s\sqrt h\xi)-\nabla e_{n+1}(x)]^\top\xi\,ds\right|
 \le C\sqrt h\|\nabla e_{n+1}\|_\infty,
\]
and the already proved gradient estimate \eqref{eq:grad}, shows that the
projected numerical-error remainder is uniformly $O(h^{3/2})$ on the
support.  Hence the total remainder is bounded pointwise by $Ch^{3/2}$;
the stated conditional $L^2$ estimate follows immediately.  Constants in these pointwise estimates depend on the fixed law $\mu$, in
particular through its smallest atom probability, but not on $h$ or $n$.
\end{proof}

Let $\bar B^h$ and $\bar N^h$ be the c\`adl\`ag step interpolations,
constant on $[t_n,t_{n+1})$, and put $M^h=\bar N^h/\sqrt h$.

\begin{lemma}[Predictable characteristics]\label{lem:char-fclt}
Set $\gamma(t,w)=\Gamma_\mu(D^2u(t,w))$.
\begin{equation}\label{eq:bracket-ucp}
 \sup_{0\le t\le T}
 \left|
 \langle M^h\rangle_t-
 \int_0^t\gamma(s,\bar B_s^h)\,ds
 \right|\xrightarrow{P}0.
\end{equation}
Moreover,
\[
 \langle\bar B^{h,i},\bar B^{h,j}\rangle_t
 =\delta_{ij}\sum_{t_{n+1}\le t}h,
 \qquad
 \langle\bar B^h,M^h\rangle_t\equiv0,
\]
and
\[
 \sup_{0\le t\le T}
 \bigl(|\Delta\bar B_t^h|+|\Delta M_t^h|\bigr)\longrightarrow0
\]
deterministically.  In particular, the joint martingale array
$(\bar B^h,M^h)$ satisfies the conditional Lindeberg condition.
\end{lemma}
\begin{proof}
Lemma~\ref{lem:local-gkw-fclt} gives, uniformly in $n$,
\[
 \E_n[(\Delta N_{n+1}^h)^2]
 =h^2\Gamma_\mu(A_n)+O(h^{5/2}),
 \qquad A_n=D^2u(t_{n+1},B_n^h).
\]
After division by $h$ and summation over at most $T/h$ steps, the
accumulated remainder is $O(\sqrt h)$.  Since $D^2u$ and $D^2u_t$ are bounded and $\Gamma_\mu$ is a quadratic
form, $\gamma(t,w)$ is uniformly Lipschitz in $t$, uniformly in $w$.
The piecewise-constant state $\bar B_s^h=B_n^h$ on each mesh interval
therefore gives a stepwise Riemann-sum error of order $O(h)$, uniformly in
$t$:
\[
 \sum_{t_{n+1}\le t}h\gamma(t_{n+1},B_n^h)
 -\int_0^t\gamma(s,\bar B_s^h)\,ds=O(h),
\]
which proves \eqref{eq:bracket-ucp}.  The state bracket follows from
isotropy, while the cross bracket vanishes exactly by the GKW
orthogonality
$\E_n[\Delta N_{n+1}^h\Delta B_{n+1}^h]=0$.
Finally, finite support, bounded derivatives and
Lemma~\ref{lem:local-gkw-fclt} give
$|\Delta M_{n+1}^h|+|\Delta B_{n+1}^h|\le C\sqrt h$ uniformly, proving
the vanishing-jump and conditional Lindeberg assertions.
\end{proof}

\begin{lemma}[Uniqueness of the limiting martingale problem]\label{lem:unique-fclt}
Set
\[
 \gamma(t,w)=\Gamma_\mu(D^2u(t,w)),\qquad
 \sigma(t,w)=\sqrt{\gamma(t,w)}.
\]
The martingale problem with generator
\[
 \mathcal L_t\varphi(w,m)
 =\frac12\Delta_w\varphi(w,m)
 +\frac12\gamma(t,w)\partial_{mm}\varphi(w,m)
\]
and initial state $(0,0)$ has a unique law, namely
\[
 \left(W_t,\int_0^t\sigma(s,W_s)\,dW_s^\perp\right)_{0\le t\le T},
\]
where $W^\perp$ is a one-dimensional Brownian motion independent of $W$.
\end{lemma}
\begin{proof}
Because $A\mapsto\mathcal R_A$ is linear,
\[
 |\sqrt{\Gamma_\mu(A)}-\sqrt{\Gamma_\mu(B)}|
 \le\|\mathcal R_{A-B}\|_2\le C|A-B|.
\]
Assumption~\ref{ass:Hnum} bounds $D^3u$, hence $D^2u(t,\cdot)$ is
uniformly Lipschitz and $\sigma(t,\cdot)$ is bounded and uniformly
Lipschitz.  Therefore the triangular SDE
\[
 dW_t=d\widetilde W_t,\qquad
 dM_t=\sigma(t,W_t)\,dW_t^\perp,
\]
with independent Brownian drivers $(\widetilde W,W^\perp)$ has pathwise
uniqueness and a unique weak law.  Its diffusion matrix is
$\operatorname{diag}(I_d,\gamma(t,w))$, and the associated martingale
problem is therefore well posed; see, for example, Ethier and
Kurtz~\cite[Chapter~5]{EthierKurtz1986} for the standard correspondence
between weak uniqueness for SDEs and uniqueness of the associated
martingale problem.
\end{proof}

\begin{theorem}[Joint functional limit of the orthogonal martingale]\label{thm:N-fclt}
Under Assumptions~\ref{ass:increment} and~\ref{ass:Hnum},
\begin{equation}\label{eq:N-fclt}
 \left(\bar B^h,\frac{\bar N^h}{\sqrt h}\right)
 \Longrightarrow
 \left(W,\int_0^.\sqrt{\Gamma_\mu(D^2u(s,W_s))}\,dW_s^\perp\right)
\end{equation}
in $D([0,T];\R^{d+1})$ with the $J_1$ topology.
\end{theorem}
\begin{proof}
Set $M^h=\bar N^h/\sqrt h$ and consider the space--time process
\[
 \mathcal X_t^h=(t,\bar B_t^h,M_t^h).
\]
The drift characteristics of the last $d+1$ coordinates vanish.  By
Lemma~\ref{lem:char-fclt}, their predictable covariance characteristics
converge uniformly on $[0,T]$ in probability to those associated with the
space--time generator
\[
 \mathcal A\varphi(t,w,m)
 =\partial_t\varphi(t,w,m)
 +\frac12\Delta_w\varphi(t,w,m)
 +\frac12\Gamma_\mu(D^2u(t,w))\partial_{mm}\varphi(t,w,m).
\]
Indeed, the state bracket differs from $tI_d$ by at most $hI_d$, the
state--orthogonal cross bracket is identically zero, and
\eqref{eq:bracket-ucp} gives the remaining component.  The maximal jump of
$\mathcal X^h$ tends to zero, so the conditional Lindeberg condition also
holds.

The diffusion-approximation theorem for martingale arrays, formulated in
terms of convergence of predictable local characteristics and the
Lindeberg condition (see Ethier and Kurtz~\cite[Chapter~7, Theorem~4.1]{EthierKurtz1986},
together with the martingale-problem convergence framework of
Chapter~4), therefore gives tightness and shows that every subsequential
limit solves the martingale problem for $\mathcal A$.  Equivalently, this
is the time-homogeneous space--time formulation of the martingale problem
for $\mathcal L_t$ in Lemma~\ref{lem:unique-fclt}.  That lemma gives
well-posedness and identifies the unique law as
\[
 \left(W,\int_0^.\sqrt{\Gamma_\mu(D^2u(s,W_s))}\,dW_s^\perp\right).
\]
Hence the full sequence converges as claimed.  In particular, the
vanishing drift, convergence of predictable covariance characteristics,
conditional Lindeberg condition, and uniqueness of the limiting martingale
problem provide the required diffusion-approximation ingredients.
\end{proof}

\begin{remark}[Classical discretization-error limits and finite-state geometry]
Independent Gaussian innovations and mixed-normal limits are familiar in
classical discretization-error asymptotics for stochastic integrals and
hedging; see Rootz\'en~\cite{Rootzen1980}, Jacod and
Protter~\cite{JacodProtter1998}, Hayashi and
Mykland~\cite{HayashiMykland2005}, and Fukasawa~\cite{Fukasawa2014}.
The new feature here is the dependence of the limiting conditional variance
on the \emph{finite-state one-step law} through the quadratic GKW defect
$\Gamma_\mu$.  Theorem~\ref{thm:universal-defect} and
Corollary~\ref{cor:defect-extremizers} give a sharp worst-direction lower bound
on the trace-free quadratic sector and characterize its finite-state
extremizers.  For comparison, a standard Gaussian increment has
$\Gamma_G(A)=\tfrac12|A|_F^2$ for trace-free $A$, whereas the sharp
finite-state worst-direction bound is $d/[2(d+2)]$.
This comparison concerns the defect functional, not a claim that a
finite-state increment is globally preferable to a Gaussian one.
The present result is formulated as joint weak convergence with the state
process; it retains the Brownian-path dependence of the random limiting
variance without requiring a common-space embedding of the approximations.
\end{remark}

\begin{corollary}[Terminal mixed-normal limit]\label{cor:N-mixed}
With
\[
 V_T=\int_0^T\Gamma_\mu(D^2u(s,W_s))\,ds,
\]
we have
\[
 \frac{N_N^h}{\sqrt h}\Longrightarrow\sqrt{V_T}\,G,
\]
where, on an extension of the probability space if necessary,
$G\sim N(0,1)$ is independent of $W$.
\end{corollary}

For the centrally symmetric simplex, if $I$ is uniform on
$\{0,\ldots,d\}$, Proposition~\ref{prop:quad-defect} reduces to
\begin{equation}\label{eq:Gamma-sym}
 \Gamma_{\rm sym}(A)
 =\frac14\operatorname{Var}_{I}
 (\widehat v_I^\top A\widehat v_I).
\end{equation}
Indeed,
$M_4[A,A]=(d+1)^{-1}\sum_i(\widehat v_i^\top A\widehat v_i)^2$ and the
simplex covariance identity makes their mean equal to $\operatorname{tr}A$.
In particular,
\begin{equation}\label{eq:Gamma-rank-one}
 \Gamma_{\rm sym}(\widehat v_0\widehat v_0^\top)
 =\frac{d(d-1)^2}{4}.
\end{equation}
Thus the limiting orthogonal noise measures the dispersion of directional
curvatures across the simplex directions.  Formula~\eqref{eq:Gamma-rank-one}
will recover the exact Fourier variance constant below.

\section{Sharpness}

We first show that the difference between the minimal-simplex and symmetric-simplex rules
is visible already for the driver-free heat equation.

\subsection{Minimal simplex: a genuine square-root term}

Fix \(a=\widehat v_0\).  For the minimal simplex law,
\[
 X:=a^\top\xi^{\rm min}
 =\begin{cases}
 d,&\text{with probability }1/(d+1),\\
 -1,&\text{with probability }d/(d+1).
 \end{cases}
\]
Consequently
\[
 \E[X^2]=d,\qquad
 \E[X^3]=d(d-1),\qquad
 \E[X^4]=d(d^2-d+1).
\]
For terminal datum
\[
 g_\theta(x)=\sin(a^\top x+\theta),
\]
the Brownian heat solution at time zero is
\[
 u(0,x)=e^{-dT/2}\sin(a^\top x+\theta).
\]
The minimal-simplex one-step multiplier is
\[
 m_{\rm min}(h)
 =\frac{e^{id\sqrt h}+d e^{-i\sqrt h}}{d+1}.
\]
Its logarithm satisfies
\[
 \log m_{\rm min}(h)
 =-\frac d2h
 -i\frac{d(d-1)}6h^{3/2}
 +c_4h^2+O(h^{5/2}),
\]
where
\[
 c_4=\frac{d(d^2-4d+1)}{24}.
\]
Thus, with \(N=T/h\),
\[
 m_{\rm min}(h)^N
 =e^{-dT/2}\left[
 1-i\frac{d(d-1)T}{6}\sqrt h+O(h)\right].
\]
At \(x=0\),
\begin{equation}\label{eq:minsharp-full}
 U_0^{\rm min}(0)-u(0,0)
 =-e^{-dT/2}\frac{d(d-1)T}{6}
   \cos\theta\,\sqrt h+O(h).
\end{equation}
For \(d\ge2\) and \(\cos\theta\ne0\), the error is of exact order
\(\sqrt h\).

\subsection{Symmetric simplex: the odd term disappears}

For the symmetric law,
\[
 m_{\rm sym}(h)
 =\frac{\cos(d\sqrt h)+d\cos(\sqrt h)}{d+1},
\]
so
\[
 \log m_{\rm sym}(h)
 =-\frac d2h+c_4h^2+O(h^3).
\]
Hence
\begin{equation}\label{eq:symsharp-full}
 U_0^{\rm sym}(0)-u(0,0)
 =e^{-dT/2}c_4T\sin\theta\,h+O(h^2).
\end{equation}
Since $d^2-4d+1$ has no integer root, $c_4\ne0$ for integer $d\ge2$.
Thus whenever \(\sin\theta\ne0\), the \(O(h)\) rate is sharp.  The comparison
\eqref{eq:minsharp-full}--\eqref{eq:symsharp-full} isolates the exact
effect of central symmetry.

\subsection{Sharpness of the orthogonal martingale scale}

The \(O(\sqrt h)\) maximal \(L^2\) scale of \(N^h\) is already attained by
the Fourier benchmark.

\begin{proposition}[Exact orthogonal-martingale asymptotic]\label{prop:Nexact}
For the symmetric-simplex Fourier benchmark
\(g_\theta(x)=\sin(\hat v_0^\top x+\theta)\), \(d\ge2\),
\[
 \lim_{h\downarrow0}\frac1h\E[|N_N^h|^2]
 =
 C_{d,T,\theta},
\]
where
\[
 C_{d,T,\theta}
 =
 \frac{(d-1)^2}{8}
 (1-e^{-dT})
 (1-e^{-dT}\cos(2\theta))
 >0.
\]
Consequently
\[
 \|N_N^h\|_{L^2}
 =
 \sqrt{C_{d,T,\theta}}\,\sqrt h+o(\sqrt h).
\]
\end{proposition}

\begin{proof}
Put \(a=\hat v_0\) and \(X=a^\top\xi\).  Under the symmetric-simplex law,
\(X=\pm d\) with total probability \(1/(d+1)\) and \(X=\pm1\) with total
probability \(d/(d+1)\).  Thus
\[
 Q:=X^2-d
\]
is centered and, since the third moment vanishes, orthogonal to the full
linear span of \(\xi\).  Moreover
\[
 \E[Q^2]=d(d-1)^2.
\]
Writing
\[
 m_h=\frac{\cos(d\sqrt h)+d\cos(\sqrt h)}{d+1},\quad
 \Lambda_{n+1}=m_h^{N-n-1},\quad
 \Phi_n=a^\top B_n^h+\theta,
\]
the one-step GKW residual satisfies
\[
 \Delta N_{n+1}^h
 =
 -\frac h2\Lambda_{n+1}\sin(\Phi_n)Q+O_{L^2}(h^{3/2}).
\]
Summing the conditional variances gives
\[
 \frac1h\E[|N_N^h|^2]
 =
 \frac{d(d-1)^2}{4}
 \sum_{n=0}^{N-1}
 h\Lambda_{n+1}^2\E[\sin^2(\Phi_n)]
 +O(\sqrt h).
\]
We spell out the passage from the random phases to the deterministic
Riemann sum.  Since the symmetric-simplex law is centrally symmetric, its
one-step characteristic function in the direction $a$ is real.  At twice
the Fourier frequency, put
\[
 r_h:=\E[e^{2i\sqrt h X}]
 =\frac{\cos(2d\sqrt h)+d\cos(2\sqrt h)}{d+1}.
\]
By independence of the increments of $B^h$,
\[
 \E[e^{2i\Phi_n}]
 =e^{2i\theta}r_h^n,
\]
and hence
\begin{equation}\label{eq:sin2-phase-expectation}
 \E[\sin^2(\Phi_n)]
 =\frac12\{1-r_h^n\cos(2\theta)\}.
\end{equation}
This identity is exact for every $h$ and $n$; in particular, the random
quantity $\sin^2(\Phi_n)$ is replaced by a deterministic expression before
any Riemann-sum argument is used.

Taylor expansion gives
\[
 m_h=1-\frac d2h+O(h^2),\qquad
 r_h=1-2dh+O(h^2).
\]
Consequently,
\[
 \log m_h=-\frac d2h+O(h^2),\qquad
 \log r_h=-2dh+O(h^2).
\]
Since $N=T/h$ and $0\le n\le N$, the accumulated remainders are uniform in
$n$.  With $t_n=nh$ this yields
\begin{equation}\label{eq:uniform-multiplier-limits}
 \sup_{0\le n<N}
 \left|\Lambda_{n+1}^2-e^{-d(T-t_{n+1})}\right|=O(h),
 \qquad
 \sup_{0\le n\le N}
 \left|r_h^n-e^{-2dt_n}\right|=O(h).
\end{equation}
(The harmless shift from $t_{n+1}$ to $t_n$ also changes the first term by
only $O(h)$, uniformly in $n$.)  Combining
\eqref{eq:sin2-phase-expectation} and
\eqref{eq:uniform-multiplier-limits},
\[
 \sup_{0\le n<N}\left|
 \Lambda_{n+1}^2\E[\sin^2(\Phi_n)]
 -\frac12e^{-d(T-t_n)}
   \{1-e^{-2dt_n}\cos(2\theta)\}
 \right|=O(h).
\]
Therefore the triangular array in the preceding display differs by $o(1)$
from an ordinary deterministic Riemann sum, and
\begin{align*}
 &\frac{d(d-1)^2}{4}
 \sum_{n=0}^{N-1}h\Lambda_{n+1}^2\E[\sin^2(\Phi_n)]\\
 &\qquad\longrightarrow
 \frac{d(d-1)^2}{8}
 \int_0^T e^{-d(T-t)}
 \{1-e^{-2dt}\cos(2\theta)\}\,dt.
\end{align*}
Finally,
\[
 \int_0^T e^{-d(T-t)}\,dt=\frac{1-e^{-dT}}d,
\qquad
 \int_0^T e^{-d(T-t)}e^{-2dt}\,dt
 =\frac{e^{-dT}(1-e^{-dT})}{d},
\]
so the last display equals
\[
 \frac{(d-1)^2}{8}(1-e^{-dT})
 \{1-e^{-dT}\cos(2\theta)\}
 =C_{d,T,\theta}.
\]
\end{proof}

The same constant is also recovered from the functional limit.  Indeed,
Lemma~\ref{lem:local-gkw-fclt}, finite support and the discrete BDG
inequality give
\[
 \sup_h\E\!\left[\left|\frac{N_N^h}{\sqrt h}\right|^4\right]<\infty,
\]
so the squared terminal variables are uniformly integrable.  Hence
Corollary~\ref{cor:N-mixed} yields
\[
 \frac1h\E[|N_N^h|^2]
 \longrightarrow
 \E[\int_0^T\Gamma_{\rm sym}(D^2u(s,W_s))\,ds].
\]
Using \eqref{eq:Gamma-rank-one} and the Fourier Hessian
$D^2u(s,x)=-e^{-d(T-s)/2}\sin(a^\top x+\theta)aa^\top$ reproduces
$C_{d,T,\theta}$ above.  Thus the direct finite-step calculation provides
an independent check of the general functional limit.

The proposition makes the mechanism explicit.  The quadratic Taylor term
lies in the GKW complement at order \(h\); the conditional variances are
therefore of order \(h^2\), and their sum over \(T/h\) steps is of order
\(h\).  This is compatible with first-order accuracy of the value field:
the orthogonal martingale measures the one-step component not representable
by the \(d\)-dimensional increment, rather than a bias in the Markov value
recursion.

\section{Predictable representation versus third-moment cancellation}

The preceding tradeoff is not specific to the regular simplex.

Exact scalar one-step PRP means
\[
 L_0^2(\xi)
 =\{a^\top\xi:a\in\R^d\},
\]
where \(L_0^2(\xi)\) is the centered scalar one-step function space.
By Proposition~\ref{prop:minprp}, this already forces minimal support
\(m=d+1\).  The next theorem shows that this algebraic completeness is
incompatible with the moment cancellation needed to remove the generic
square-root weak term.

\begin{remark}[Relation with discrete structure equations]\label{rem:obtuse}
For a centered normalized random variable taking exactly \(d+1\) values,
the functions \(1,\xi_1,\ldots,\xi_d\) form an orthonormal basis of the
one-step \(L^2\) space.  Hence each product \(\xi_i\xi_j\) has an expansion
\[
 \xi_i\xi_j=\delta_{ij}+\sum_{k=1}^d S^{ij}_k\xi_k,
\]
the discrete structure equation of the obtuse-random-variable literature.
Moreover
\[
 S^{ij}_k=\E[\xi_i\xi_j\xi_k].
\]
Thus Theorem~\ref{thm:prp} can also be read as the elementary statement that
the structure tensor of a nondegenerate \(d\)-dimensional minimal-state law
cannot vanish when \(d\ge2\).  We include a direct proof because it is short and gives exactly the form
needed for the weak-error argument.
\end{remark}

\begin{lemma}[Quadratic structure coefficients under exact PRP]
\label{lem:quadratic-structure}
Let \(\xi\) be centered and isotropic and assume exact scalar one-step PRP.
For \(1\le i,j\le d\), define the centered quadratic variable
\[
 Q_{ij}:=\xi_i\xi_j-\delta_{ij}.
\]
Then there exists a unique vector \(c^{ij}\in\R^d\) such that
\begin{equation}\label{eq:qij-prp}
 Q_{ij}=(c^{ij})^\top\xi
 \qquad\text{in }L^2.
\end{equation}
Its coordinates are exactly the third-order moments:
\begin{equation}\label{eq:qij-coeff}
 c^{ij}_k=\E[\xi_i\xi_j\xi_k],\qquad k=1,\ldots,d.
\end{equation}
Equivalently,
\[
 \xi_i\xi_j=\delta_{ij}+\sum_{k=1}^d\E[\xi_i\xi_j\xi_k]\,\xi_k
 \qquad\text{in }L^2.
\]
\end{lemma}

\begin{proof}
Because \(\E[\xi_i\xi_j]=\delta_{ij}\), the variable \(Q_{ij}\) is centered.
Exact scalar PRP therefore gives \eqref{eq:qij-prp}.  The representing
coefficient is unique: if \(c^\top\xi=0\) in \(L^2\), then isotropy gives
\[
 0=\E[(c^\top\xi)^2]=c^\top\E[\xi\xi^\top]c=|c|^2,
\]
hence \(c=0\).

To identify the coefficient, multiply \eqref{eq:qij-prp} by \(\xi_k\) and
take expectations.  Using centeredness and isotropy,
\[
 \E[Q_{ij}\xi_k]
 =\E[\xi_i\xi_j\xi_k]-\delta_{ij}\E[\xi_k]
 =\E[\xi_i\xi_j\xi_k],
\]
whereas
\[
 \E[((c^{ij})^\top\xi)\xi_k]
 =\sum_{\ell=1}^d c^{ij}_\ell\E[\xi_\ell\xi_k]
 =c^{ij}_k.
\]
This proves \eqref{eq:qij-coeff}.
\end{proof}

\begin{theorem}[PRP--third-moment obstruction]\label{thm:prp}
Let \(\xi\) be a finite-support random vector in \(\R^d\) with positive
atom probabilities,
\[
 \E[\xi]=0,\qquad \E[\xi\xi^\top]=I_d.
\]
Assume exact scalar one-step predictable representation: every centered
scalar \(L^2\)-function of \(\xi\) is a linear function of \(\xi\).  If
\[
 \E[\xi^{\otimes3}]=0,
\]
then \(d=1\).  In particular, for \(d\ge2\), exact scalar PRP and vanishing
third moment are incompatible.
\end{theorem}

\begin{proof}
Assume for contradiction that \(d\ge2\) and that
\(\E[\xi^{\otimes3}]=0\).  By Lemma~\ref{lem:quadratic-structure}, all the
coefficient vectors \(c^{ij}\) vanish.  Hence, for every
\(1\le i,j\le d\),
\begin{equation}\label{eq:quadratic-collapse}
 \xi_i\xi_j=\delta_{ij}\qquad\text{almost surely}.
\end{equation}
Taking \(i=j=1\) and \(i=j=2\) gives
\[
 \xi_1^2=1,\qquad \xi_2^2=1\qquad\text{almost surely}.
\]
Thus \(|\xi_1\xi_2|=1\) almost surely.  On the other hand, taking
\(i=1\), \(j=2\) in \eqref{eq:quadratic-collapse} gives
\[
 \xi_1\xi_2=0\qquad\text{almost surely},
\]
a contradiction.  Therefore \(d\) cannot exceed one.

For \(d=1\), the centered isotropic Bernoulli law
\(\mathbb P(\xi=1)=\mathbb P(\xi=-1)=1/2\) has exact scalar one-step PRP and
vanishing third moment, so the exceptional one-dimensional case is genuine.
\end{proof}

Theorem~\ref{thm:universal-defect} gives a quantitative quadratic counterpart
to this obstruction.  Exact scalar PRP forces every centered quadratic
function into the linear sector and hence forces \(\Gamma_\mu(A)=0\) for
all \(A\).  Under third-moment cancellation in dimension \(d\ge2\),
Theorem~\ref{thm:universal-defect} instead gives a strictly positive universal
lower bound on the trace-free quadratic sector.  No converse implication
from vanishing quadratic defect to exact PRP is used here.

\begin{remark}[Equivalent directional formulation]\label{rem:directional-prp}
The same structure can be written in the directional form used later in the
weak-error calculation.  For \(a\in\R^d\),
\[
 Q_a:=(a^\top\xi)^2-|a|^2
\]
is centered and exact PRP gives \(Q_a=c(a)^\top\xi\).  Isotropy implies
\[
 c(a)^\top b
 =\E[(a^\top\xi)^2(b^\top\xi)]
 =\E[\xi^{\otimes3}][a,a,b].
\]
The directional identity is an algebraic consequence of the coordinate
identities in Lemma~\ref{lem:quadratic-structure}; equivalently,
\[
 Q_a=\sum_{i,j=1}^d a_i a_j Q_{ij}.
\]
Thus the third-order tensor is precisely the coefficient map of the
centered quadratic sector.  If the tensor vanished, the finite family
\(Q_{ij}\) would vanish simultaneously almost surely, and hence
\(Q_a=0\) for every \(a\).  Conversely,
if exact PRP holds in dimension \(d\ge2\), Theorem~\ref{thm:prp} implies
that the symmetric cubic form
\[
 a\longmapsto \mu_3(a):=\E[(a^\top\xi)^3]
\]
is not identically zero.  Hence some \(a\) satisfies \(\mu_3(a)\ne0\).
\end{remark}

\begin{corollary}[A universal first-order barrier under exact PRP]
\label{cor:barrier}
Let \(d\ge2\) and let \(\xi\) be a finite-support centered isotropic
increment with exact scalar one-step PRP.  Then there exist \(a\in\R^d\) and \(\theta\in\R\) such that, for the bounded
\(C^\infty\) terminal function \(g(x)=\sin(a^\top x+\theta)\), the
associated heat random-walk approximation satisfies
\[
 |U_0^h-u(0,0)|=\Theta(\sqrt h).
\]
Consequently, for the one-step random-walk propagation considered here,
no such exact-PRP increment law can satisfy a universal \(O(h)\) weak-error
bound over all bounded smooth Fourier terminal data.
\end{corollary}

\begin{proof}
By Theorem~\ref{thm:prp} and Remark~\ref{rem:directional-prp}, the cubic
form is not identically zero, so there exists \(a\) with
\[
 \mu_3(a)=\E[(a^\top\xi)^3]\ne0.
\]
Throughout \(h=T/N\) with \(N\in\mathbb N\).  Write
\[
 m(h):=\E[e^{i\sqrt h\,a^\top\xi}]
 =1-\frac{|a|^2}{2}h-\frac{i\mu_3(a)}6h^{3/2}+O(h^2).
\]
For all sufficiently small \(h\), \(m(h)\) lies in a fixed neighborhood of
\(1\), so we use the analytic logarithm there.  Then
\[
 \log m(h)
 =-\frac{|a|^2}{2}h-\frac{i\mu_3(a)}6h^{3/2}+O(h^2),
\]
and hence
\[
 m(h)^{T/h}
 =e^{-|a|^2T/2}
 \left(1-\frac{i\mu_3(a)T}{6}\sqrt h+O(h)\right).
\]
For \(g(x)=\sin(a^\top x+\theta)\),
\[
 U_0^h=\Im\!\left(e^{i\theta}m(h)^{T/h}\right),\qquad
 u(0,0)=e^{-|a|^2T/2}\sin\theta.
\]
Therefore
\[
 U_0^h-u(0,0)
 =-e^{-|a|^2T/2}\frac{\mu_3(a)T}{6}\cos\theta\,\sqrt h+O(h).
\]
Choosing, for instance, \(\theta=0\) makes the leading coefficient nonzero.
\end{proof}

\begin{remark}[Scope and time-inhomogeneous laws]
\label{rem:scope-obstruction}
The global Fourier barrier uses the same law at every step.  This restriction is
essential: if \(\eta_n\) are i.i.d.\ minimal-simplex increments and
\(\xi_n=(-1)^n\eta_n\), each step has exact scalar PRP, while successive
Fourier multipliers are conjugate.  For an even number of steps the cubic
phase defects cancel and the heat value error is \(O(h)\).

The one-step obstruction remains.  For smooth \(\phi\),
\[
 D_h\phi-\nabla\phi
 =\frac{\sqrt h}{2}\E[\xi D^2\phi[\xi,\xi]]+O(h).
\]
Thus every exact-PRP law in \(d\ge2\) has a Hessian direction with a local
\(\sqrt h\) predictable-field defect; for the alternating example only its
sign changes.  We make no corresponding global \(Z\)-barrier claim for an
arbitrary time-dependent sequence, since one fixed Hessian need not detect
all of its third-order tensors.  The theorem is therefore a fixed-law
accuracy--representation obstruction, not an impossibility result for
general complete BSDE algorithms.
\end{remark}

\section{Numerical validation}\label{sec:numerics}

The purpose of this section is deliberately limited.  The experiments do
not establish the convergence theorem; rather, they illustrate the sharp
asymptotic distinction predicted by the analysis and check that it persists
for a genuinely semilinear driver.  No Monte Carlo error is present.

\subsection{Exact Fourier benchmark}

Take \(T=1\), \(\theta=0.37\), \(a=\widehat v_0\), and
\[
 g(x)=\sin(a^\top x+\theta),\qquad f=0.
\]
The exact value is \(u(0,0)=e^{-d/2}\sin\theta\).  For the minimal and
symmetric simplex walks the corresponding characteristic multipliers are
\[
 m_{\min}(h)=\frac{e^{id\sqrt h}+d e^{-i\sqrt h}}{d+1},
 \qquad
 m_{\rm sym}(h)=\frac{\cos(d\sqrt h)+d\cos(\sqrt h)}{d+1}.
\]
Hence the numerical values at time zero are evaluated directly from
\(m(h)^{T/h}\), without simulation.  We use
\(N=128,256,512,1024,2048,4096\); Table~\ref{tab:fourier-order} reports
least-squares slopes over the five finest meshes.

\begin{table}[ht]
\centering
\caption{Observed weak orders in the exact Fourier benchmark.}
\label{tab:fourier-order}
\begin{tabular}{c|cc}
\(d\) & minimal simplex & symmetric simplex\\ \hline
2 & 0.506 & 1.000\\
3 & 0.506 & 1.000\\
5 & 0.509 & 1.003
\end{tabular}
\end{table}

The experiment reproduces the \(\sqrt h\) slope of the minimal-simplex
Fourier barrier in Section~7.1 and Corollary~\ref{cor:barrier}, together
with the \(h\) slope predicted by Theorem~\ref{thm:main} for the
third-moment-free symmetric rule.  The order separation is already visible at moderate mesh sizes.

\subsection{A semilinear manufactured solution}

We next test nonlinear \(y\)-dependence and nontrivial \(z\)-dependence while
retaining an exact reference solution.  Let
\[
 u(t,x)=e^{-\kappa(T-t)}\sin(a^\top x+\theta),\qquad
 f(t,x,y,z)=\lambda\tanh(y)+\beta^\top z+q(t,x),
\]
where \(\lambda=0.6\), \(\kappa=0.35\), and \(\beta\) is parallel to \(a\)
with \(\beta^\top a=0.12\).  We define
\[
 q(t,x)=
 -\partial_tu(t,x)-\frac12\Delta u(t,x)
 -\lambda\tanh(u(t,x))-\beta^\top\nabla u(t,x),
\]
so that \(u\) solves the semilinear PDE exactly.

The solution depends only on \(r=a^\top x\), so this is an exact
directional stress test rather than a genuinely multidirectional numerical
example; the multidimensional content of Theorem~\ref{thm:main} is
analytic.  We implement the periodic one-dimensional recursion spectrally:
translations are applied through their Fourier multipliers, while the scalar
implicit \(y\)-equation is solved by Newton iteration to machine precision.
There is no Monte Carlo error and no off-grid interpolation.  We use a
periodic spectral grid with \(M=1024\) points; doubling the resolution to
\(M=2048\) leaves all reported slopes unchanged to three decimal places.
We use \(N=16,32,64,128,256,512\).  The reported value error is
\(|U_0(0)-u(0,0)|\), and the reported field error is
\[
 \left|D_hU_1(0)-\nabla u(t_1,0)\right|,
\]
with the Euclidean norm for the vector quantity.  Thus Table~\ref{tab:nonlinear-order}
is a pointwise numerical check, not a numerical estimate of the sup-norm
bounds in Theorem~\ref{thm:main}.  The table reports least-squares slopes
over the four finest meshes.

\begin{table}[ht]
\centering
\caption{Observed orders for the semilinear manufactured test.}
\label{tab:nonlinear-order}
\begin{tabular}{c|cc|cc}
& \multicolumn{2}{c|}{value \(Y\)} & \multicolumn{2}{c}{field \(Z\)}\\
\(d\)& minimal & symmetric & minimal & symmetric\\ \hline
2&0.509&1.005&0.599&1.001\\
3&0.504&1.004&0.565&1.003\\
5&0.508&1.007&0.549&1.007
\end{tabular}
\end{table}

The symmetric rule is first order for both quantities to numerical
precision.  The minimal rule has the predicted half-order value error.
For \(Z\), the finite-mesh minimal-rule slopes are above one half and decrease
as the mesh is refined.  The one-step expansion in
Remark~\ref{rem:scope-obstruction} shows why a \(\sqrt h\) local
predictable-field defect is generically present for an exact-PRP law, but
Table~\ref{tab:nonlinear-order} is not used to claim a separate global
half-order theorem for \(Z\).

\section{Structural consequences and open questions}

The comparison distinguishes three properties whose compatibility depends
on the finite-state increment law:
covariance matching, cancellation of the cubic weak defect, and predictable
representation.  Exact scalar one-step PRP fixes the dimension of the
centered \(L^2\) space at \(d\).  In dimension \(d\ge2\), that minimal
geometry necessarily leaves a nonzero third-order structure tensor.  By
contrast, in dimension $d\ge2$, third-moment-free rules require additional
one-step degrees of freedom, and these degrees of freedom become the GKW-orthogonal sector of
the BS\(\Delta\)E.  For the two simplex rules compared explicitly in this
paper, the order-\(h^{1/2}\) effect therefore appears in different places:
as a weak bias for the minimal exact-PRP rule, and as the scale of the
orthogonal martingale for the symmetric first-order rule.  This observation
is specific to the comparison established here and is not asserted as a
universal conservation principle for finite-state schemes.

From a financial perspective, the obstruction may also be read as a design
principle for finite-state approximations of multidimensional value and
first-order risk quantities.  In a Markovian BSDE, the value field
represents the continuation value, while its spatial gradient and
predictable \(Z\)-component carry first-order sensitivity and
martingale-exposure information.  Within the finite-state class studied
here, exact one-step martingale representation and the third-moment
cancellation underlying first-order smooth approximation cannot in general
be imposed simultaneously.  Thus the increment law should be chosen
according to whether the numerical design prioritizes exact one-step
risk representation or higher accuracy for smooth value and sensitivity
quantities.

The symmetric simplex is only one realization of the first-order class.
For example, the \(2d\)-point axis rule in
Remark~\ref{rem:axisrule} has the same vanishing-third-moment property.  We
have used the symmetric simplex for the sharp analysis because its
antipodal relationship with the minimal regular simplex makes the
representation tradeoff and the Fourier calculation explicit.  With
respect to the operator defect in \eqref{eq:worst-defect-lb},
Corollary~\ref{cor:defect-extremizers} and
Remark~\ref{rem:total-directional-defect} show that the symmetric simplex
is optimal in $d=2$ and is not optimal for $d\ge3$.  We make no optimality
claim for it with respect to support size, higher moments, or computational
complexity.  A related open question is the tradeoff between support size
and quadratic defect, including the minimal support size of a weighted
spherical $4$-design attaining \eqref{eq:worst-defect-lb}.

The one-dimensional case is exceptional.  There the centrally symmetric
simplex reduces to the ordinary two-point Bernoulli law, so exact
predictable representation and odd-moment cancellation coexist.  The
obstruction is therefore genuinely multidimensional.

The functional limit in Section~\ref{sec:fclt} shows that the quadratic
GKW sector is precisely the source of the leading orthogonal fluctuation:
after normalization by $h^{-1/2}$, the residual converges jointly with the
state process to an independent-innovation stochastic integral whose
random variance is $\int\Gamma_\mu(D^2u)\,dt$.  Natural extensions include
state-dependent increment laws, nonsmooth terminal data, and higher-order
finite-state schemes in which fourth and higher Gaussian moments are also
matched.

\appendix

\section{Residual estimates}

We record the consistency estimates used in the proof of
Theorem~\ref{thm:main}.  This also makes explicit why the regularity package
stops at fifth spatial derivatives.

\begin{lemma}[Value and differentiated residuals]\label{lem:residuals}
Under Assumptions~\ref{ass:increment} and \ref{ass:Hnum}, let
\(u_n=u(t_n,\cdot)\), \(p_n=\nabla u(t_n,\cdot)\), and define
\[
 z_n^\ast=D_hu_{n+1},\qquad
 \Theta_n^\ast=(t_{n+1},x,u_n,z_n^\ast).
\]
Then
\[
 \tau_n^u
 =u_n-P_hu_{n+1}-h f(\Theta_n^\ast)
\]
satisfies
\[
 \|\tau_n^u\|_\infty\le Ch^2.
\]
Moreover, with
\[
 \tau_n^p
 =
 p_n-P_hp_{n+1}
 -h\left\{
 f_x(\Theta_n^\ast)
 +f_y(\Theta_n^\ast)p_n
 +(J_hp_{n+1})f_z(\Theta_n^\ast)
 \right\},
\]
where \(J_hp_{n+1}=\nabla(D_hu_{n+1})\), one has
\[
 \|\tau_n^p\|_\infty\le Ch^2.
\]
\end{lemma}

\begin{proof}
We first isolate the linear heat part.  Proposition~\ref{prop:local} gives
\[
 P_hu_{n+1}
 =
 u_{n+1}+\frac h2\Delta u_{n+1}+R_n^P,
 \qquad
 \|R_n^P\|_\infty
 \le \frac{\kappa_4}{24}h^2\|D^4u\|_\infty .
\]
Taylor's formula in time yields
\[
 u_{n+1}
 =
 u_n+h(u_t)_n+R_n^t,\qquad
 \|R_n^t\|_\infty\le \frac{h^2}{2}\|u_{tt}\|_\infty ,
\]
and
\[
 \frac h2(\Delta u_{n+1}-\Delta u_n)
 =
 \frac h2\int_{t_n}^{t_{n+1}}\Delta u_t(s,\cdot)\,ds
 =O(h^2)
\]
uniformly, by the bound on \(D^2u_t\).  Hence
\begin{equation}\label{eq:app-linear-value}
 P_hu_{n+1}
 =
 u_n+h\left(u_t+\frac12\Delta u\right)_n+O_\infty(h^2).
\end{equation}
Furthermore, Proposition~\ref{prop:local} and the time regularity of the
gradient give
\begin{equation}\label{eq:app-zshift}
 \|D_hu_{n+1}-p_n\|_\infty
 \le
 \|D_hu_{n+1}-p_{n+1}\|_\infty
 +\|p_{n+1}-p_n\|_\infty
 \le Ch.
\end{equation}
Since \(f\) is Lipschitz in \(t\) and \(z\),
\[
 f(\Theta_n^\ast)
 =
 f(t_n,x,u_n,p_n)+O_\infty(h).
\]
Combining this identity with \eqref{eq:app-linear-value} and the PDE proves
\(\|\tau_n^u\|_\infty\le Ch^2\).

For the differentiated residual, apply the same heat-step calculation to
the vector field \(p=\nabla u\).  Componentwise use of
Proposition~\ref{prop:local}, followed by Taylor expansion in time, gives
\begin{equation}\label{eq:app-linear-grad}
 P_hp_{n+1}
 =
 p_n+h\left(p_t+\frac12\Delta p\right)_n+O_\infty(h^2).
\end{equation}
Here the spatial remainder uses \(D^4p=D^5u\), while the time remainder
uses \(p_{tt}=\nabla u_{tt}\) and
\(\Delta p_t=D^3u_t\).

It remains to compare the differentiated driver at \(t_{n+1}\) with the
spatial derivative of the PDE driver at \(t_n\).  By applying the
\(D_h\)-estimate in Proposition~\ref{prop:local} componentwise to
\(p_{n+1}\),
\[
 \|J_hp_{n+1}-D^2u_{n+1}\|_\infty\le Ch,
\]
where the constant uses \(D^4u\).  The time shift then gives
\[
 \|J_hp_{n+1}-D^2u_n\|_\infty\le Ch
\]
by the bound on \(D^2u_t\).  Together with
\eqref{eq:app-zshift} and the joint Lipschitz continuity of
\(f_x,f_y,f_z\), this yields
\begin{align*}
 &f_x(\Theta_n^\ast)
 +f_y(\Theta_n^\ast)p_n
 +(J_hp_{n+1})f_z(\Theta_n^\ast)\\
 &\qquad =
 f_x(t_n,x,u_n,p_n)
 +f_y(t_n,x,u_n,p_n)p_n
 +(D^2u_n)f_z(t_n,x,u_n,p_n)
 +O_\infty(h).
\end{align*}
The first three terms on the right are precisely
\[
 \nabla_x\bigl[f(t_n,x,u_n,\nabla u_n)\bigr].
\]
Differentiating the PDE in space therefore gives
\[
 p_t+\frac12\Delta p
 +f_x+f_yp+(D^2u)f_z=0.
\]
Combining this identity with \eqref{eq:app-linear-grad} proves
\(\|\tau_n^p\|_\infty\le Ch^2\).
\end{proof}

\begin{remark}[Derivative count]
The proof uses \(D^4u\) for the value heat-step remainder and \(D^5u\) after
one spatial differentiation.  Time shifts require \(u_{tt}\),
\(\nabla u_{tt}\), \(D^2u_t\), \(D^3u_t\), and \(\nabla u_t\).
No sixth spatial derivative of \(u\) and no second time derivative of the
driver are used.
\end{remark}

\section{One-step martingale geometry}

For \(d\ge2\), label the symmetric support by \((i,\sigma)\),
\(0\le i\le d\), \(\sigma\in\{\pm1\}\).  The centered scalar function
space has dimension \(2d+1\).  The linear sector has dimension \(d\), so
the GKW complement has dimension \(d+1\).

There is a useful parity decomposition.  The even centered sector has
dimension \(d\).  In addition there is one odd direction orthogonal to all
linear coordinate functions.  Quadratic Taylor terms enter the even sector
at order \(h\), whereas the additional odd direction can first be reached
by cubic terms at order \(h^{3/2}\).  This explains why the leading
orthogonal martingale is generated by the Hessian.

For \(d=1\), this labelled count is artificial: the symmetric simplex is
just the two-point Bernoulli law, and exact PRP is recovered.

\section*{Declaration of generative AI and AI-assisted technologies in the manuscript preparation process}

During the preparation of this work, the author used OpenAI ChatGPT to
assist with literature searches, manuscript organization, language editing,
and the drafting and testing of numerical code.  All mathematical formulas,
proofs, derivations, and numerical results presented in this manuscript were
independently verified by the author, who takes full responsibility for their
accuracy and for the content of the publication.

\section*{Disclaimer}
The views expressed in this article are solely those of the author and do not necessarily reflect those of Sumitomo Mitsui Banking Corporation.

\end{document}